\documentclass[reqno]{amsart}
\usepackage{amssymb,latexsym}
\usepackage{graphicx}
\usepackage{fancyhdr}
\numberwithin{equation}{section}
\newtheorem{thm}{Theorem}[section]
\newtheorem{theorem}[thm]{Theorem}

\newtheorem{lemma}[thm]{Lemma}

\newtheorem{corollary}[thm]{Corollary}

\begin{document}

\setcounter{page}{1}

\title[Mean values]{On the mean values of products of Dirichlet $L$-functions at positive integers}
\author{Yuan He}
\address{School of Mathematics and Information Science, Neijiang Normal University, Neijiang 641100, Sichuan, People's Republic of China}
\email{hyyhe@aliyun.com}

\begin{abstract}
In this paper, we study the mean value distributions of Dirichlet $L$-functions at positive integers. We give some explicit formulas for the mean values of products of two and three Dirichlet $L$-functions at positive integers weighted by Dirichlet characters that involve the Bernoulli functions and Jordan's totient functions. The results presented here are the generalizations of various known formulas.
\end{abstract}

\maketitle

\section{Introduction}

Let $\mathbb{N}$ be the set of positive integers, $\mathbb{N}_{0}$ the set of non-negative integers, $\mathbb{Z}$ the set of integers, $\mathbb{R}$ the set of real numbers, and $\mathbb{C}$ the set of complex numbers. For $q\in\mathbb{N}$, the Dirichlet $L$-function $L(s,\chi)$ is defined for $s\in\mathbb{C}$ by the series
\begin{equation}\label{eq1.1}
L(s,\chi)=\sum_{n=1}^{\infty}\frac{\chi(n)}{n^{s}}\quad(\Re(s)>1),
\end{equation}
where $\chi$ is a Dirichlet character modulo $q$. It is well known that the values of $L(s,\chi)$ at integers, especially when $\chi$ is a non-principal
character, have good distribution properties and important arithmetical and algebraic
information; see, for example, \cite{alkan3,berndt1,berndt2,cohen,kim}.

In the present paper, we will be concerned with the mean value distributions of products of Dirichlet $L$-functions at positive integers. Perhaps the earliest result in this
direction is Paley and Selberg's \cite{selberg} asymptotic formula, namely,
\begin{equation}\label{eq1.2}
\sum_{\substack{\text{$\chi$ (mod $p$)}\\ \chi\not=\chi_{0}}}|L(1,\chi)|^{2}=\frac{\pi^{2}}{6}p+O\bigl((\log p)^{2}\bigl),
\end{equation}
where $p$ is a prime, $\chi_{0}$ is the principal
character. In 1951, Ankeny and Chowla \cite{ankeny} provided a direct simple proof of \eqref{eq1.2}. In 1986, Slavutskii \cite{slavutskii} improved the error estimate in \eqref{eq1.2} to $O(\log p)$. Further refinements were given by Zhang \cite{zhang3,zhang4} in 1990, where the error estimate in \eqref{eq1.2} is improved to $O(\log\log p)$ and $O(\frac{1}{\log p})$, respectively. See also \cite{katsurada,lee1,lee2,zhang2} for further exploration.

On the other hand, let $q,m_{1},m_{2},m_{3}\in\mathbb{N}$, $a,b,c\in\mathbb{Z}$, and set
\begin{equation}\label{eq1.3}
V_{\chi,q}(m_{1},m_{2};a,b)=\sum_{\substack{\text{$\chi$ (mod $q$)}\\ \chi(-1)=(-1)^{m_{1}}=(-1)^{m_{2}}}}\chi(a)\overline{\chi}(b)L(m_{1},\chi)L(m_{2},\overline{\chi}),
\end{equation}
and
\begin{eqnarray}\label{eq1.4}
&&V_{\chi_{1},\chi_{2},q}(m_{1},m_{2},m_{3};a,b,c)\nonumber\\
&&=\sum_{\substack{\text{$\chi_{1},\chi_{2}$ (mod $q$)}\\ \chi_{1}(-1)=(-1)^{m_{1}}\\ \chi_{2}(-1)=(-1)^{m_{2}}}}\chi_{1}(a)\chi_{2}(b)\overline{\chi_{1}\chi_{2}}(c)L(m_{1},\chi_{1})L(m_{2},\chi_{2})L(m_{3},\overline{\chi_{1}\chi_{2}}),
\end{eqnarray}
where $\overline{\chi}$ denotes the complex conjugation of $\chi$. In 1982, Walum \cite{walum} proved that if $p$ is an odd prime then
\begin{equation}\label{eq1.5}
V_{\chi,p}(1,1;1,1)=\frac{\pi^{2}(p-1)^{2}(p-2)}{12p^{2}}.
\end{equation}
About the year 1989, Zhang \cite[Theorem 3]{zhang1} extended \eqref{eq1.5}, and demonstrated that for $q\in\mathbb{N}$ with $q\geq3$,
\begin{equation}\label{eq1.6}
V_{\chi,q}(1,1;1,1)=\frac{\pi^{2}\phi(q)^{2}}{12q^{2}}\biggl(q\prod_{\substack{p\mid q\\ \text{$p$ prime}}}\biggl(1+\frac{1}{p}\biggl)-3\biggl),
\end{equation}
where $\phi(q)$ is Euler's totient function. After that, Qi \cite{qi}, Zhang \cite{zhang5,zhang6,zhang7}, Louboutin \cite{louboutin1,louboutin2} and Alkan \cite{alkan2} provided different proofs of \eqref{eq1.6}. In particular, Zhang \cite[Theorem]{zhang6} proved that for $q,m,a\in\mathbb{N}$ with $q\geq3$ and $(a,q)=1$,
\begin{equation*}
\sum_{d\mid q}\frac{d^{2n}}{\phi(d)}V_{\chi,d}(m,m;a,1)
\end{equation*}
can be explicitly expressed by the generalized Dedekind sums, and then evaluated $V_{\chi,q}(1,1;1,1),V_{\chi,q}(2,2;1,1),V_{\chi,q}(3,3;1,1)$ in terms of Jordan's totient function $J_{n}(q)$ defined for $q,n\in\mathbb{N}$ by
\begin{equation}\label{eq1.7}
J_{n}(q)=q^{n}\prod_{\substack{p\mid q\\ \text{$p$ prime}}}\biggl(1-\frac{1}{p^{n}}\biggl)\quad\bigl(\phi(q)=J_{1}(q)\bigl).
\end{equation}
Louboutin \cite[Theorem 1]{louboutin2} discovered that for $q,m\in\mathbb{N}$ with $q\geq3$, there exists a polynomial $R_{m}(x)=\sum_{l=0}^{2m}r_{m,l}x^{l}$ of degree $2m$ with rational coefficients $r_{m,l}$ such that
\begin{equation}\label{eq1.8}
V_{\chi,q}(m,m;1,1)=\frac{\pi^{2m}\phi(q)}{4q^{2m}\bigl((m-1)!\bigl)^{2}}\sum_{l=1}^{2m}r_{m,l}J_{l}(q),
\end{equation}
where $J_{l}(q)$ is as in \eqref{eq1.7}.
In 2006, Liu and Zhang \cite[Theorem 1.1]{liu1} determined the coefficients $r_{m,l}$ appearing in \eqref{eq1.8}, and obtained the following more general formula
\begin{equation}\label{eq1.9}
V_{\chi,q}(m,n;1,1)=\frac{(-1)^{\frac{m-n}{2}}(2\pi)^{m+n}\phi(q)}{4q^{m+n}m!n!}\biggl(\sum_{l=1}^{m+n}r_{m,n,l}J_{l}(q)-\epsilon_{m,n}\frac{\phi(q)}{4}\biggl),
\end{equation}
where $q,m,n\in\mathbb{N}$ with $q\geq2$, $J_{l}(q)$ is as in \eqref{eq1.7}, the coefficients $r_{m,n,l}$ are given by
\begin{equation*}
r_{m,n,l}=B_{m+n-l}\underset{j+k\geq m+n-l}{\sum_{j=0}^{m}\sum_{k=0}^{n}}B_{m-j}B_{n-k}\frac{\binom{m}{j}\binom{n}{k}\binom{j+k+1}{m+n-l}}{j+k+1}
\end{equation*}
with $B_{n}$ being the $n$-th Bernoulli number,
\begin{equation*}
\epsilon_{m,n}=\begin{cases}
1,  &\text{if $m=n=1$},\\
0,  &\text{otherwise}.
\end{cases}
\end{equation*}
For some equivalent versions and additional results of \eqref{eq1.9}, one is referred to \cite{bayad,kanemitsu,okamoto,xu}. In 2012, Wu and Zhang \cite[Corollary 1.1]{wu} shown that for an odd prime $p$ and $k\in\mathbb{N}$, $V_{\chi,p}(1,1;2^{k},1)$ can be evaluated by a recursive sequence. Further, Louboutin \cite[Theorem 1]{louboutin3} in 2014 established a more general formula for $V_{\chi,q}(1,1;a,1)$ in the case when $q,a\in\mathbb{N}$ with $q\geq3$ and $a\geq2$, which involves the cotangent sums. Liu \cite[Theorems 1.1 and 1.2]{liu2} in 2015 gave some explicit formulas for $V_{\chi,p}(1,n;a,1)$ and $V_{\chi,p}(2,n;a,1)$ in the case when $p$ is an odd prime and $n,a\in\mathbb{N}$ with $p>3$ and $1\leq a\leq4$, which involve the Bernoulli polynomials. More recently, Louboutin and Munsch \cite[Theorem 3.6 and Corollary 3.8]{louboutin5} obtained two asymptotic formulas for $V_{\chi,p}(1,1;a,b)$ in the case when $p$ is an odd prime and $a,b\in\mathbb{N}$ with $p>3$, by virtue of which some known results stated in \cite{lee1,lee2,louboutin4} are rediscovered and some asymptotic formulas for the sums considered by Sun \cite{sun} are given. In addition, using his formula for Dirichlet $L$-function at positive integers stated in \cite[Theorem 1]{alkan1} (see also \cite[Equation (1.17)]{he1}), Alkan \cite[Theorem 1]{alkan4} in 2013 evaluated $V_{\chi_{1},\chi_{2},q}(1,1,2;1,1,1)$ in terms of Jordan's totient functions, as follows,
\begin{equation}\label{eq1.10}
V_{\chi_{1},\chi_{2},q}(1,1,2;1,1,1)=\frac{\pi^{4}\phi(q)^{2}}{360q^{4}}\bigl(J_{4}(q)-5J_{2}(q)\bigl)\quad(q\geq3).
\end{equation}
Following the work of Alkan \cite{alkan4}, Okamoto and Onozuka \cite[Theorem 1.1]{okamoto} in 2015 used Louboutin's \cite[Proposition 3]{louboutin2} result for Dirichlet $L$-function at positive integers to obtain the formula for $V_{\chi_{1},\chi_{2},q}(m,n,m+n;1,1,1)$ in the case when $q,m,n\in\mathbb{N}$ with $q\geq3$, which involves the Bernoulli numbers, Jordan's totient functions and a recursive sequence.

It is natural to ask whether there exists some more general formulas for the mean values of products of two and three Dirichlet $L$-functions at positive integers weighted by Dirichlet characters. We give positive answers for the problems posed here. We explicitly evaluate \eqref{eq1.3} and \eqref{eq1.4} in terms of the Bernoulli functions and Jordan's totient functions (see Theorems \ref{thm2.1} and \ref{thm2.4} below).

This paper is organized as follows. In the second section, we state our main results, from which various known formulas are deduced as special cases. In the third section, we give some auxiliary lemmas. The fourth section concentrates on the feature that has contributed to the detailed proofs of Theorems \ref{thm2.1} and \ref{thm2.4}.

\section{Statement of main results}

For convenience, in the following we always denote by $\mathrm{i}$ the square root of $-1$ such that $\mathrm{i}^{2}=-1$, $\mu(q)$ the M\"{o}bius function, $(a,b)$ the greatest common factor of $a,b\in\mathbb{Z}$, $J_{\alpha}(q)$ the generalization of Jordan's totient function given for $\alpha\in\mathbb{C}$, $q\in\mathbb{N}$ by
\begin{equation*}
J_{\alpha}(q)=q^{\alpha}\prod_{\substack{p\mid q\\ \text{$p$ prime}}}\biggl(1-\frac{1}{p^{\alpha}}\biggl).
\end{equation*}
We also write, for $m,n\in\mathbb{N}_{0}$, $\delta_{m,n}$ as the Kronecker delta function given by
\begin{equation*}
\delta_{m,n}=\begin{cases}
1,  &\text{if $m=n$},\\
0,  &\text{if $m\not=n$};
\end{cases}
\end{equation*}
for $n\in\mathbb{N}_{0}$, $x\in\mathbb{R}$, $\overline{B}_{n}(x)$ as the $n$-th Bernoulli function given by $\overline{B}_{0}(x)=1$,
\begin{equation}\label{eq2.1}
\overline{B}_{1}(x)=\begin{cases}
B_{1}(\{x\}),  &\text{if $x\in\mathbb{R}\setminus\mathbb{Z}$},\\
0,  &\text{if $x\in\mathbb{Z}$},
\end{cases}
\quad\overline{B}_{n}(x)=B_{n}(\{x\})\quad(n\geq2),
\end{equation}
where $B_{n}(x)$ is the Bernoulli polynomial of degree $n$, $\{x\}$ denotes the fractional part of $x\in\mathbb{R}$.
We now present the mean values of products of two Dirichlet $L$-functions at positive integers weighted by Dirichlet characters as follows.

\begin{theorem}\label{thm2.1} Let $q,m,n,a,b\in\mathbb{N}$ with $(a,q)=(b,q)=1$. Then
\begin{eqnarray}\label{eq2.2}
V_{\chi,q}(m,n;a,b)&=&\frac{(-1)^{\frac{m-n}{2}}(2\pi)^{m+n}\phi(q)}{4q^{m+n}m!n!}\nonumber\\
&&\times\bigl(R(m,n,q;a,b)+R(n,m,q;b,a)+C_{1}\bigl),
\end{eqnarray}
where
\begin{eqnarray*}
R(m,n,q;a,b)&=&nb^{n-1}\sum_{j=0}^{m}\binom{m}{j}\frac{(-1)^{j}a^{m-j}}{m+n-j}\sum_{l=1}^{b}\overline{B}_{j}\biggl(\frac{al}{b}\biggl)\nonumber\\
&&\times\sum_{d\mid q}\mu\biggl(\frac{q}{d}\biggl)d^{j}\overline{B}_{m+n-j}\biggl(\frac{dl}{b}\biggl),
\end{eqnarray*}
and
\begin{equation*}
C_{1}=\frac{(-1)^{n-1}m!n!(a,b)^{m+n}B_{m+n}J_{m+n}(q)}{a^{n}b^{m}(m+n)!}-\delta_{1,m}\delta_{1,n}\frac{\phi(q)}{4}.
\end{equation*}
\end{theorem}

It follows that we show some special cases of Theorem \ref{thm2.1}. We first give the following result.

\begin{corollary}\label{cor2.2} Let $q,m,n,a\in\mathbb{N}$ with $q\geq2$ and $(a,q)=1$. Then
\begin{eqnarray}\label{eq2.3}
&&V_{\chi,q}(m,n;a,1)\nonumber\\
&&=\frac{(-1)^{\frac{m-n}{2}}(2\pi)^{m+n}\phi(q)}{4q^{m+n}m!n!}\biggl(n\sum_{j=1}^{m}\binom{m}{j}\frac{(-1)^{j}a^{m-j}}{m+n-j}\overline{B}_{j}(0)\overline{B}_{m+n-j}(0)J_{j}(q)\nonumber\\
&&\quad+ma^{m-1}\sum_{j=1}^{n}\binom{n}{j}\frac{(-1)^{j}}{m+n-j}\sum_{l=1}^{a}\overline{B}_{j}\biggl(\frac{l}{a}\biggl)\sum_{d\mid q}\mu\biggl(\frac{q}{d}\biggl)d^{j}\overline{B}_{m+n-j}\biggl(\frac{dl}{a}\biggl)\nonumber\\
&&\quad+\frac{(-1)^{n-1}m!n!B_{m+n}J_{m+n}(q)}{a^{n}(m+n)!}-\delta_{1,m}\delta_{1,n}\frac{\phi(q)}{4}\biggl).
\end{eqnarray}
\end{corollary}

\begin{proof}
The famous Raabe's \cite{raabe} multiplication formula claims that for $n,a\in\mathbb{N}$, $x\in\mathbb{R}$,
\begin{equation}\label{eq2.4}
a^{n-1}\sum_{l=0}^{a-1}\overline{B}_{n}\biggl(x+\frac{l}{a}\biggl)
=\overline{B}_{n}(ax).
\end{equation}
Since for $a\in\mathbb{N}$, $d\in\mathbb{Z}$ with $(a,d)=1$, $dl$ runs over a complete residue
system modulo $a$ as $l$ does, we obtain from \eqref{eq2.4} that for $n,a\in\mathbb{N}$, $d\in\mathbb{Z}$, $x\in\mathbb{R}$ with $(a,d)=1$,
\begin{equation}\label{eq2.5}
a^{n-1}\sum_{l=0}^{a-1}\overline{B}_{n}\biggl(x+\frac{dl}{a}\biggl)
=\overline{B}_{n}(ax).
\end{equation}
Note that for $q\in\mathbb{N}$, (see, e.g., \cite[Theorem 2.1]{apostol})
\begin{equation}\label{eq2.6}
\sum_{d\mid q}\mu(d)=
\begin{cases}
1,  &\text{if $q=1$},\\
0,  &\text{if $q\geq2$}.
\end{cases}
\end{equation}
More generally, we see from the property of a multiplicative function stated in \cite[Theorem 2.18]{apostol} that for $q\in\mathbb{N}$, $\alpha\in\mathbb{C}$,
\begin{equation}\label{eq2.7}
\sum_{d\mid q}\mu\biggl(\frac{q}{d}\biggl)d^{\alpha}=q^{\alpha}\sum_{d\mid q}\mu(d)\frac{1}{d^{\alpha}}=J_{\alpha}(q).
\end{equation}
Therefore, taking $b=1$ in Theorem \ref{thm2.1}, in light of \eqref{eq2.5}, \eqref{eq2.6} and \eqref{eq2.7}, we get the desired result.
\end{proof}

It becomes obvious from $\overline{B}_{1}(0)=0$ and $B_{2}=1/6$ that the case $m=n=1$ in Corollary \ref{cor2.2} gives that for $q,a\in\mathbb{N}$ with $q\geq2$ and $(a,q)=1$,
\begin{eqnarray}\label{eq2.8}
V_{\chi,q}(1,1;a,1)
&=&\frac{\pi^{2}\phi(q)^{2}}{12aq^{2}}\biggl(q\prod_{\substack{p\mid q\\ \text{$p$ prime}}}\biggl(1+\frac{1}{p}\biggl)-3a\biggl)\nonumber\\
&&-\frac{\pi^{2}\phi(q)}{q^{2}}\sum_{d\mid q}\mu\biggl(\frac{q}{d}\biggl)d\sum_{l=1}^{a}\overline{B}_{1}\biggl(\frac{l}{a}\biggl)\overline{B}_{1}\biggl(\frac{dl}{a}\biggl),
\end{eqnarray}
which is equivalent to Louboutin's \cite[Theorem 1]{louboutin3} formula.
We remark that Corollary \ref{cor2.2} extends the results of Liu \cite{liu2}, Das and Juyal \cite{das}, and provides a positive answer for the problem posed by Louboutin \cite[Section 4]{louboutin3}. In fact, we can also use Corollary \ref{cor2.2} to give an equivalent version of Liu and Zhang's formula \eqref{eq1.9} in the following way.

\begin{corollary}\label{cor2.3} Let $q,m,n\in\mathbb{N}$ with $q\geq2$. Then
\begin{eqnarray}\label{eq2.9}
&&V_{\chi,q}(m,n;1,1)\nonumber\\
&&=\frac{(-1)^{\frac{m-n}{2}}(2\pi)^{m+n}\phi(q)}{4q^{m+n}m!n!}\biggl(n\sum_{j=1}^{[\frac{m}{2}]}\binom{m}{2j}\frac{B_{2j}\overline{B}_{m+n-2j}(0)}{m+n-2j}J_{2j}(q)\nonumber\\
&&\quad+m\sum_{j=1}^{[\frac{n}{2}]}\binom{n}{2j}\frac{B_{2j}\overline{B}_{m+n-2j}(0)}{m+n-2j}J_{2j}(q)\nonumber\\
&&\quad+\frac{(-1)^{n-1}m!n!B_{m+n}J_{m+n}(q)}{(m+n)!}-\delta_{1,m}\delta_{1,n}\frac{\phi(q)}{4}\biggl),
\end{eqnarray}
where $[x]$ denotes the floor function (also called the greatest integer function) defined for $x\in\mathbb{R}$ by
\begin{equation*}
[x]=x-\{x\},
\end{equation*}
the summation in the right hand side of \eqref{eq2.9} vanishes when $m=1$ or $n=1$.
\end{corollary}

\begin{proof}
By setting $a=1$ in Corollary \ref{cor2.2} and then using \eqref{eq2.7}, we have
\begin{eqnarray}\label{eq2.10}
&&V_{\chi,q}(m,n;1,1)\nonumber\\
&&=\frac{(-1)^{\frac{m-n}{2}}(2\pi)^{m+n}\phi(q)}{4q^{m+n}m!n!}\biggl(n\sum_{j=1}^{m}\binom{m}{j}\frac{(-1)^{j}}{m+n-j}\overline{B}_{j}(0)\overline{B}_{m+n-j}(0)J_{j}(q)\nonumber\\
&&\quad+m\sum_{j=1}^{n}\binom{n}{j}\frac{(-1)^{j}}{m+n-j}\overline{B}_{j}(0)\overline{B}_{m+n-j}(0)J_{j}(q)\nonumber\\
&&\quad+\frac{(-1)^{n-1}m!n!B_{m+n}J_{m+n}(q)}{(m+n)!}-\delta_{1,m}\delta_{1,n}\frac{\phi(q)}{4}\biggl).
\end{eqnarray}
Since $\{-x\}=1-\{x\}$ in the case when $x\in\mathbb{R}\setminus\mathbb{Z}$, so by $B_{n}(1-x)=(-1)^{n}B_{n}(x)$ for $n\in\mathbb{N}_{0}$ and $x\in\mathbb{R}$ and $B_{2n+1}=0$ for $n\in\mathbb{N}$ (see, e.g., \cite[pp. 804--805]{abramowitz}), we see from \eqref{eq2.1} that for $n\in\mathbb{N}_{0}$ and $x\in\mathbb{R}$,
\begin{equation}\label{eq2.11}
\overline{B}_{n}(-x)=(-1)^{n}\overline{B}_{n}(x).
\end{equation}
Thus, by applying \eqref{eq2.11}, $\overline{B}_{1}(0)=0$ and $B_{2n+1}=0$ for $n\in\mathbb{N}$ to the right hand side of \eqref{eq2.10}, the desired result follows immediately.
\end{proof}

Clearly, the case $m=n=1$ in Corollary \ref{cor2.3} leads to Zhang's formula \eqref{eq1.6}. And the case $m=n$ in Corollary \ref{cor2.3} also yields that for $q,n\in\mathbb{N}$ with $q,n\geq2$,
\begin{eqnarray}\label{eq2.12}
V_{\chi,q}(n,n;1,1)
&=&\frac{(2\pi)^{2n}\phi(q)}{4q^{2n}(n!)^{2}}\biggl(2n\sum_{j=1}^{[\frac{n}{2}]}\binom{n}{2j}\frac{B_{2j}B_{2n-2j}}{2n-2j}J_{2j}(q)\nonumber\\
&&-\frac{(-1)^{n}(n!)^{2}B_{2n}J_{2n}(q)}{(2n)!}\biggl).
\end{eqnarray}
In particular, the case $n=2,3$ in \eqref{eq2.12} gives that for $q\geq2$,
\begin{equation}\label{eq2.13}
V_{\chi,q}(2,2;1,1)=\frac{\pi^{4}\phi(q)}{180q^{4}}\bigl(J_{4}(q)+10J_{2}(q)\bigl),
\end{equation}
and
\begin{equation}\label{eq2.14}
V_{\chi,q}(3,3;1,1)=\frac{\pi^{6}\phi(q)}{1890q^{6}}\bigl(J_{6}(q)-21J_{2}(q)\bigl).
\end{equation}
We here mention that the formulas \eqref{eq2.13} and \eqref{eq2.14} were first obtained by Zhang \cite[Corollaries 2 and 3]{zhang6}, and were rediscovered by Louboutin \cite[Theorem 2]{louboutin2} and Alkan \cite[Theorem 2]{alkan2}.

We next give the mean values of products of three Dirichlet $L$-functions at positive integers weighted by Dirichlet characters as follows.

\begin{theorem}\label{thm2.4} Let $q,m_{1},m_{2},m_{3},a,b,c\in\mathbb{N}$ with $q,m_{3}\geq2$, $m_{1}+m_{2}\equiv m_{3}$ (mod $2$) and $(a,q)=(b,q)=(c,q)=1$. Then
\begin{eqnarray}\label{eq2.15}
&&V_{\chi_{1},\chi_{2},q}(m_{1},m_{2},m_{3};a,b,c)\nonumber\\
&&=\frac{(-1)^{\frac{m_{1}-m_{2}-m_{3}}{2}}(2\pi)^{m_{1}+m_{2}+m_{3}}(b,c)^{m_{2}+m_{3}}\phi(q)^{2}}{8q^{m_{1}+m_{2}+m_{3}}b^{m_{3}}c^{m_{2}}m_{1}!(m_{2}+m_{3})!}\nonumber\\
&&\qquad\times\biggl(R\biggl(m_{1},m_{2}+m_{3},q;a,\frac{bc}{(b,c)}\biggl)+R\biggl(m_{2}+m_{3},m_{1},q;\frac{bc}{(b,c)},a\biggl)\biggl)\nonumber\\
&&\quad-\frac{(-1)^{\frac{m_{1}+m_{2}-m_{3}}{2}}(2\pi)^{m_{1}+m_{2}+m_{3}}\phi(q)^{2}}{8q^{m_{1}+m_{2}+m_{3}}m_{1}!m_{2}!m_{3}!}\nonumber\\
&&\qquad\times \bigl(A_{m_{1},m_{2},m_{3},q}(a,b,c)+(-1)^{m_{1}}A_{m_{1},m_{3},m_{2},q}(a,c,b)\nonumber\\
&&\qquad\quad+B_{m_{1},m_{2},m_{3},q}(a,b,c)+(-1)^{m_{1}}B_{m_{1},m_{3},m_{2},q}(a,c,b)\nonumber\\
&&\qquad\quad+C_{m_{1},m_{2},m_{3},q}(a,b,c)+(-1)^{m_{1}}C_{m_{1},m_{3},m_{2},q}(a,c,b)\nonumber\\
&&\qquad\quad-D_{m_{1},m_{2},m_{3},q}(b,c)+D_{m_{1},m_{3},m_{2},q}(c,b)+C_{2}\bigl),
\end{eqnarray}
where $R(m_{1},m_{2},q;a,b)$ is as in \eqref{eq2.2},
\begin{eqnarray*}
&&A_{m_{1},m_{2},m_{3},q}(a,b,c)\\
&&=m_{3}b^{m_{2}-1}c^{m_{3}-1}\sum_{j=1}^{m_{2}}\binom{m_{2}}{j}\frac{(-1)^{j}j}{m_{2}+m_{3}-j}\sum_{j_{1}=0}^{m_{1}}\binom{m_{1}}{j_{1}}
\frac{(-1)^{j_{1}}a^{m_{1}-j_{1}}}{m_{1}+j-j_{1}}\\
&&\quad\times\sum_{l=1}^{c}\sum_{l_{1}=1}^{b}\overline{B}_{j_{1}}\biggl(\frac{al_{1}}{b}+\frac{al}{c}\biggl)\sum_{d\mid q}\mu\biggl(\frac{q}{d}\biggl)d^{j_{1}}\overline{B}_{m_{2}+m_{3}-j}\biggl(\frac{dl}{c}\biggl)\nonumber\\
&&\quad\times\overline{B}_{m_{1}+j-j_{1}}\biggl(\frac{dl_{1}}{b}+\frac{dl}{c}\biggl),
\end{eqnarray*}
\begin{eqnarray*}
&&B_{m_{1},m_{2},m_{3},q}(a,b,c)\\
&&=m_{1}m_{2}a^{m_{1}-1}b^{m_{2}-1}\sum_{j=1}^{m_{3}}\binom{m_{3}}{j}\frac{1}{m_{2}+m_{3}-j}\sum_{j_{1}=0}^{j}\binom{j}{j_{1}}\frac{(-1)^{j_{1}}c^{m_{3}-j_{1}}}{m_{1}+j-j_{1}}\\
&&\quad\times\sum_{l=1}^{b}\sum_{l_{1}=1}^{a}\overline{B}_{j_{1}}\biggl(\frac{cl_{1}}{a}+\frac{cl}{b}\biggl)\sum_{d\mid q}\mu\biggl(\frac{q}{d}\biggl)d^{j_{1}}\overline{B}_{m_{2}+m_{3}-j}\biggl(\frac{dl}{b}\biggl)
\overline{B}_{m_{1}+j-j_{1}}\biggl(\frac{dl_{1}}{a}\biggl),
\end{eqnarray*}
\begin{eqnarray*}
&&C_{m_{1},m_{2},m_{3},q}(a,b,c)\\
&&=-\frac{m_{2}m_{1}!b^{m_{2}-1}(a,c)^{m_{1}}}{c^{m_{1}}}\sum_{j=1}^{m_{3}}\binom{m_{3}}{j}\frac{(-1)^{j}(a,c)^{j}c^{m_{3}-j}j!}{a^{j}(m_{2}+m_{3}-j)(m_{1}+j)!}\\
&&\quad\times\sum_{l=1}^{b}\overline{B}_{m_{1}+j}\biggl(\frac{\frac{ac}{(a,c)}l}{b}\biggl)\sum_{d\mid q}\mu\biggl(\frac{q}{d}\biggl)d^{m_{1}+j}\overline{B}_{m_{2}+m_{3}-j}\biggl(\frac{dl}{b}\biggl),
\end{eqnarray*}
\begin{eqnarray*}
&&D_{m_{1},m_{2},m_{3},q}(b,c)\\
&&=\frac{\delta_{1,m_{1}}m_{2}m_{3}b^{m_{2}-1}c^{m_{3}-1}}{4(m_{2}+m_{3}-1)}
\sum_{d\mid q}\mu\biggl(\frac{q}{d}\biggl)d^{m_{1}}\sum_{l=1}^{b}\overline{B}_{m_{2}+m_{3}-1}\biggl(\frac{dl}{b}\biggl)\delta_{\mathbb{Z}}\biggl(\frac{cl}{b}\biggl),
\end{eqnarray*}
\begin{equation*}
C_{2}=\frac{(-1)^{m_{1}+m_{2}}m_{1}!m_{2}!m_{3}!\bigl(a(b,c),bc\bigl)^{m_{1}+m_{2}+m_{3}}B_{m_{1}+m_{2}+m_{3}}J_{m_{1}+m_{2}+m_{3}}(q)}{a^{m_{2}+m_{3}}b^{m_{1}+m_{3}}
c^{m_{1}+m_{2}}(m_{1}+m_{2}+m_{3})!}.
\end{equation*}
\end{theorem}

It is evident that the case $a=b=c=1$ and $m_{3}=m_{1}+m_{2}$ in Theorem \ref{thm2.4} improves Okamoto and Onozuka's \cite[Theorem 1.1]{okamoto} formula.
In particular, we have the following results.

\begin{corollary}\label{cor2.5} Let $q,n\in\mathbb{N}$ with $q\geq2$ and $2\mid n$. Then
\begin{equation}\label{eq2.16}
V_{\chi_{1},\chi_{2},q}(1,1,n;1,1,1)=\frac{(-1)^{\frac{n}{2}}(2\pi)^{n+2}\phi(q)^{2}}{8q^{n+2}n!}(M_{1}+M_{2}+M_{3}),
\end{equation}
where
\begin{equation*}
M_{1}=2\sum_{j=1}^{\frac{n}{2}}\binom{n}{2j}\frac{B_{2j}B_{n+2-2j}J_{2j}(q)}{(n+2-2j)(n+1-2j)},
\end{equation*}
\begin{equation*}
M_{2}=\sum_{j=3}^{n-1}\binom{n}{j}\frac{B_{n+1-j}}{n+1-j}\sum_{l=2}^{j-1}\binom{j}{l}\frac{B_{l}
B_{j+1-l}J_{l}(q)}{j+1-l},
\end{equation*}
\begin{equation*}
M_{3}=-\frac{B_{2}B_{n}J_{2}(q)}{2}+\frac{B_{n+2}J_{n+2}(q)}{(n+2)(n+1)},
\end{equation*}
in which $M_{2}$ vanishes when $n=2$.
\end{corollary}

\begin{proof}
Taking $m_{1}=m_{2}=a=b=c=1$ in Theorem \ref{thm2.4}, in view of \eqref{eq2.6}, \eqref{eq2.7} and \eqref{eq2.11}, we know that for $q,n\in\mathbb{N}$ with $q\geq2$ and $2\mid n$,
\begin{eqnarray}\label{eq2.17}
&&V_{\chi_{1},\chi_{2},q}(1,1,n;1,1,1)\nonumber\\
&&=\frac{(-1)^{\frac{n}{2}}(2\pi)^{n+2}\phi(q)^{2}}{8q^{n+2}n!}\biggl(\frac{R(n+1,1,q;1,1)}{n+1}+B_{1,1,n,q}(1,1,1)\nonumber\\
&&\quad+C_{1,1,n,q}(1,1,1)-\frac{B_{2}B_{n}J_{2}(q)}{2}+\frac{B_{n+2}J_{n+2}(q)}{(n+2)(n+1)}\biggl),
\end{eqnarray}
where
\begin{equation*}
R(n+1,1,q;1,1)=\sum_{j=2}^{n}\binom{n+1}{j}\frac{B_{j}B_{n+2-j}J_{j}(q)}{n+2-j},
\end{equation*}
\begin{equation*}
B_{1,1,n,q}(1,1,1)
=\sum_{j=3}^{n-1}\binom{n}{j}\frac{B_{n+1-j}}{n+1-j}\sum_{l=2}^{j-1}\binom{j}{l}\frac{B_{l}
B_{j+1-l}J_{l}(q)}{j+1-l},
\end{equation*}
\begin{equation*}
C_{1,1,n,q}(1,1,1)
=\sum_{j=1}^{n-1}\binom{n}{j}\frac{B_{j+1}B_{n+1-j}J_{j+1}(q)}{(n+1-j)(j+1)}.
\end{equation*}
Observe that
\begin{eqnarray}\label{eq2.18}
&&\frac{R(n+1,1,q;1,1)}{n+1}+C_{1,1,n,q}(1,1,1)\nonumber\\
&&=\sum_{j=2}^{n}\binom{n}{j}\frac{B_{j}B_{n+2-j}J_{j}(q)}{(n+2-j)(n+1-j)}+\sum_{j=2}^{n}\binom{n}{j-1}\frac{B_{j}B_{n+2-j}J_{j}(q)}{(n+2-j)j}\nonumber\\
&&=2\sum_{j=2}^{n}\binom{n}{j}\frac{B_{j}B_{n+2-j}J_{j}(q)}{(n+2-j)(n+1-j)}.
\end{eqnarray}
Therefore, by inserting \eqref{eq2.18} into \eqref{eq2.17}, in view of $B_{2n+1}=0$ for $n\in\mathbb{N}$, we get the desired result.
\end{proof}

It is easy to check that the case $n=2$ in Corollary \ref{cor2.5} gives Alkan's formula \eqref{eq1.10}. If we take $n=4$ in Corollary \ref{cor2.5} then we have
\begin{equation}\label{eq2.19}
V_{\chi_{1},\chi_{2},q}(1,1,4;1,1,1)=\frac{\pi^{6}\phi(q)^{2}}{3780q^{4}}\bigl(J_{6}(q)-7J_{4}(q)+14J_{2}(q)\bigl)\quad(q\geq2).
\end{equation}

\begin{corollary}\label{cor2.6} Let $q,n\in\mathbb{N}$ with $q,n\geq2$ and $2\nmid n$. Then
\begin{equation}\label{eq2.20}
V_{\chi_{1},\chi_{2},q}(1,2,n;1,1,1)=\frac{(-1)^{\frac{n+1}{2}}(2\pi)^{n+3}\phi(q)^{2}}{8q^{n+3}n!}(N_{1}+N_{2}+N_{3}),
\end{equation}
where
\begin{equation*}
N_{1}=-\sum_{j=1}^{\frac{n-1}{2}}\binom{n}{2j}\frac{B_{2j}B_{n+3-2j}J_{2j}(q)}{(n+3-2j)(n+2-2j)},
\end{equation*}
\begin{equation*}
N_{2}=-\sum_{j=3}^{n}\binom{n}{j}\frac{B_{n+2-j}}{n+2-j}\sum_{l=2}^{j-1}\binom{j}{l}\frac{B_{l}
B_{j+1-l}J_{l}(q)}{j+1-l},
\end{equation*}
\begin{equation*}
N_{3}=\frac{nB_{2}B_{n+1}J_{2}(q)}{2(n+1)}+\frac{B_{n+3}J_{n+3}(q)}{(n+3)(n+2)(n+1)}.
\end{equation*}
\end{corollary}

\begin{proof}
Setting $m_{1}=a=b=c=1$ and $m_{2}=2$ in Theorem \ref{thm2.4}, in light of \eqref{eq2.6}, \eqref{eq2.7} and \eqref{eq2.11}, we conclude that for $q,n\in\mathbb{N}$ with $q,n\geq2$ and $2\nmid n$,
\begin{eqnarray}\label{eq2.21}
&&V_{\chi_{1},\chi_{2},q}(1,2,n;1,1,1)\nonumber\\
&&=\frac{(-1)^{\frac{n+1}{2}}(2\pi)^{n+3}\phi(q)^{2}}{8q^{n+3}n!}\biggl(\frac{R(n+2,1,q;1,1)}{(n+2)(n+1)}-B_{1,2,n,q}(1,1,1)\nonumber\\
&&\quad-C_{1,2,n,q}(1,1,1)+\frac{nB_{2}B_{n+1}J_{2}(q)}{2(n+1)}+\frac{B_{n+3}J_{n+3}(q)}{(n+3)(n+2)(n+1)}\biggl),
\end{eqnarray}
where
\begin{equation*}
R(n+2,1,q;1,1)=\sum_{j=2}^{n+1}\binom{n+2}{j}\frac{B_{j}B_{n+3-j}J_{j}(q)}{n+3-j},
\end{equation*}
\begin{equation*}
B_{1,2,n,q}(1,1,1)
=\sum_{j=3}^{n}\binom{n}{j}\frac{B_{n+2-j}}{n+2-j}\sum_{l=2}^{j-1}\binom{j}{l}\frac{B_{l}
B_{j+1-l}J_{l}(q)}{j+1-l},
\end{equation*}
\begin{equation*}
C_{1,2,n,q}(1,1,1)
=\sum_{j=1}^{n}\binom{n}{j}\frac{B_{j+1}B_{n+2-j}J_{j+1}(q)}{(n+2-j)(j+1)}.
\end{equation*}
Since
\begin{eqnarray}\label{eq2.22}
&&\frac{R(n+2,1,q;1,1)}{(n+2)(n+1)}-C_{1,2,n,q}(1,1,1)\nonumber\\
&&=\frac{1}{n+1}\sum_{j=2}^{n+1}\binom{n+1}{j}\frac{B_{j}B_{n+3-j}J_{j}(q)}{(n+3-j)(n+2-j)}-\sum_{j=2}^{n+1}\binom{n}{j-1}\frac{B_{j}B_{n+3-j}J_{j}(q)}{(n+3-j)j}\nonumber\\
&&=-\sum_{j=2}^{n}\binom{n}{j}\frac{B_{j}B_{n+3-j}J_{j}(q)}{(n+3-j)(n+2-j)},
\end{eqnarray}
by inserting \eqref{eq2.22} into \eqref{eq2.21} and then using $B_{2n+1}=0$ for $n\in\mathbb{N}$, the desired result follows immediately.
\end{proof}

It is trivial to see that the case $n=3$ in Corollary \ref{cor2.6} gives
\begin{equation}\label{eq2.23}
V_{\chi_{1},\chi_{2},q}(1,2,3;1,1,1)=\frac{\pi^{6}\phi(q)^{2}}{3780q^{6}}\bigl(J_{6}(q)-21J_{2}(q)\bigl)\quad(q\geq2),
\end{equation}
which is due to Okamoto and Onozuka \cite[p. 113]{okamoto}.
If we take $n=5$ in Corollary \ref{cor2.6} then we have
\begin{equation}\label{eq2.24}
V_{\chi_{1},\chi_{2},q}(1,2,5;1,1,1)=\frac{\pi^{8}\phi(q)^{2}}{37800q^{8}}\bigl(J_{8}(q)-7J_{4}(q)-50J_{2}(q)\bigl)\quad(q\geq2).
\end{equation}
In a similar consideration to Corollaries \ref{cor2.5} and \ref{cor2.6}, one can also easily extend Okamoto and Onozuka's \cite[pp. 113--114]{okamoto} another two formulas, and give the explicit formulas for $V_{\chi_{1},\chi_{2},q}(1,3,n;1,1,1)$ and $V_{\chi_{1},\chi_{2},q}(2,2,n;1,1,1)$ for $q,n\in\mathbb{N}$ with $q,n\geq2$ and $2\mid n$. This is left to the interested reader for an exercise.

\section{Some auxiliary results}

Before giving our proofs of Theorems \ref{thm2.1} and \ref{thm2.4}, we need the following auxiliary lemmas. We first present the following result.

\begin{lemma}\label{lem3.1} Let $m,n\in\mathbb{N}$, $a,b\in\mathbb{Z}$ with $a\not=0$ and $b\not=0$. Then, for $x,y,z\in\mathbb{R}$,
\begin{eqnarray}\label{eq3.1}
&&\overline{B}_{m}(ax+y)\overline{B}_{n}(bx+z)\nonumber\\
&&=nb^{n-1}\mathrm{sgn}(b)\sum_{j=0}^{m}\binom{m}{j}\frac{(-1)^{j}a^{m-j}}{m+n-j}\sum_{l=1}^{|b|}\overline{B}_{j}\biggl(\frac{a(l+z)}{b}-y\biggl)\nonumber\\
&&\qquad\times\overline{B}_{m+n-j}\biggl(x+\frac{l+z}{b}\biggl)\nonumber\\
&&\quad+ma^{m-1}\mathrm{sgn}(a)\sum_{j=0}^{n}\binom{n}{j}\frac{(-1)^{j}b^{n-j}}{m+n-j}\sum_{l=1}^{|a|}\overline{B}_{j}\biggl(\frac{b(l+y)}{a}-z\biggl)\nonumber\\
&&\qquad\times\overline{B}_{m+n-j}\biggl(x+\frac{l+y}{a}\biggl)\nonumber\\
&&\quad+\frac{(-1)^{n-1}m!n!(a,b)^{m+n}\overline{B}_{m+n}\bigl(\frac{by-az}{(a,b)}\bigl)}{a^{n}b^{m}(m+n)!}\nonumber\\
&&\quad-\delta_{1,m}\delta_{1,n}\mathrm{sgn}(ab)\frac{\delta_{\mathbb{Z}}(ax+y)\delta_{\mathbb{Z}}(bx+z)}{4},
\end{eqnarray}
where $\delta_{\mathbb{Z}}(x)=1$ or $0$ according to $x\in\mathbb{Z}$ or $x\not\in\mathbb{Z}$, $\mathrm{sgn}(x)$ is the sign of $x\in\mathbb{R}\setminus\{0\}$ given by $\mathrm{sgn}(x)=x/|x|$.
\end{lemma}

\begin{proof}
See \cite[Theorem 3.1]{he2} for details.
\end{proof}

We next use Lemma \ref{lem3.1} to product the following results.

\begin{lemma}\label{lem3.2} Let $q,m,n,a,b\in\mathbb{N}$ with $(a,q)=(b,q)=1$. Then
\begin{eqnarray}\label{eq3.2}
&&\sum_{d\mid q}\mu\biggl(\frac{q}{d}\biggl)d^{m+n-1}\sum_{j=0}^{d-1}\overline{B}_{m}\biggl(\frac{aj}{d}\biggl)\overline{B}_{n}\biggl(\frac{bj}{d}\biggl)\nonumber\\
&&=R(m,n,q;a,b)+R(n,m,q;b,a)+C_{1},
\end{eqnarray}
where $R(m,n,q;a,b)$ and $C_{1}$ are as in \eqref{eq2.2}.
\end{lemma}

\begin{proof}
Substituting $k/d$ for $x$, $0$ for $y$, $0$ for $z$, and then making the operation $\sum_{k=0}^{d-1}$ on both sides of \eqref{eq3.1}, we obtain that for $d,m,n,a,b\in\mathbb{N}$ with $(a,d)=(b,d)=1$,
\begin{eqnarray*}
&&\sum_{k=0}^{d-1}\overline{B}_{m}\biggl(\frac{ak}{d}\biggl)\overline{B}_{n}\biggl(\frac{bk}{d}\biggl)\\
&&=nb^{n-1}\sum_{j=0}^{m}\binom{m}{j}\frac{(-1)^{j}a^{m-j}}{m+n-j}\sum_{l=1}^{b}\overline{B}_{j}\biggl(\frac{al}{b}\biggl)
\sum_{k=0}^{d-1}\overline{B}_{m+n-j}\biggl(\frac{k}{d}+\frac{l}{b}\biggl)\\
&&\quad+ma^{m-1}\sum_{j=0}^{n}\binom{n}{j}\frac{(-1)^{j}b^{n-j}}{m+n-j}\sum_{l=1}^{a}\overline{B}_{j}\biggl(\frac{bl}{a}\biggl)
\sum_{k=0}^{d-1}\overline{B}_{m+n-j}\biggl(\frac{k}{d}+\frac{l}{a}\biggl)\\
&&\quad+\frac{(-1)^{n-1}dm!n!(a,b)^{m+n}B_{m+n}}{a^{n}b^{m}(m+n)!}-\frac{\delta_{1,m}\delta_{1,n}}{4},
\end{eqnarray*}
from which and \eqref{eq2.4} we see that for $q,m,n,a,b\in\mathbb{N}$ with $(a,q)=(b,q)=1$,
\begin{eqnarray}\label{eq3.3}
&&\sum_{d\mid q}\mu\biggl(\frac{q}{d}\biggl)d^{m+n-1}\sum_{j=0}^{d-1}\overline{B}_{m}\biggl(\frac{aj}{d}\biggl)\overline{B}_{n}\biggl(\frac{bj}{d}\biggl)\nonumber\\
&&=nb^{n-1}\sum_{j=0}^{m}\binom{m}{j}\frac{(-1)^{j}a^{m-j}}{m+n-j}\sum_{l=1}^{b}\overline{B}_{j}\biggl(\frac{al}{b}\biggl)\sum_{d\mid q}\mu\biggl(\frac{q}{d}\biggl)d^{j}\overline{B}_{m+n-j}\biggl(\frac{dl}{b}\biggl)\nonumber\\
&&\quad+ma^{m-1}\sum_{j=0}^{n}\binom{n}{j}\frac{(-1)^{j}b^{n-j}}{m+n-j}\sum_{l=1}^{a}\overline{B}_{j}\biggl(\frac{bl}{a}\biggl)\sum_{d\mid q}\mu\biggl(\frac{q}{d}\biggl)d^{j}\overline{B}_{m+n-j}\biggl(\frac{dl}{a}\biggl)\nonumber\\
&&\quad+\frac{(-1)^{n-1}m!n!(a,b)^{m+n}B_{m+n}}{a^{n}b^{m}(m+n)!}\sum_{d\mid q}\mu\biggl(\frac{q}{d}\biggl)d^{m+n}\nonumber\\
&&\quad-\frac{\delta_{1,m}\delta_{1,n}}{4}\sum_{d\mid q}\mu\biggl(\frac{q}{d}\biggl)d^{m+n-1}.
\end{eqnarray}
Thus, by applying \eqref{eq2.7} to the right hand side of \eqref{eq3.3}, we prove Lemma \ref{lem3.2}.
\end{proof}

\begin{lemma}\label{lem3.3} Let $q,m_{1},m_{2},m_{3},a,b,c\in\mathbb{N}$ with $q,m_{3}\geq2$ and $(a,q)=(b,q)=(c,q)=1$. Then
\begin{eqnarray}\label{eq3.4}
&&\sum_{d\mid q}\mu\biggl(\frac{q}{d}\biggl)d^{m_{1}+m_{2}+m_{3}-2}\sum_{s=0}^{d-1}\sum_{t=0}^{d-1}\overline{B}_{m_{1}}
\biggl(\frac{as}{d}\biggl)\overline{B}_{m_{2}}\biggl(\frac{bt}{d}\biggl)\overline{B}_{m_{3}}\biggl(\frac{c(s+t)}{d}\biggl)\nonumber\\
&&=\frac{(-1)^{m_{2}-1}m_{2}!m_{3}!(b,c)^{m_{2}+m_{3}}}{b^{m_{3}}c^{m_{2}}(m_{2}+m_{3})!}\nonumber\\
&&\qquad\times\biggl(R\biggl(m_{1},m_{2}+m_{3},q;a,\frac{bc}{(b,c)}\biggl)+R\biggl(m_{2}+m_{3},m_{1},q;\frac{bc}{(b,c)},a\biggl)\biggl)\nonumber\\
&&\quad+A_{m_{1},m_{2},m_{3},q}(a,b,c)+(-1)^{m_{2}+m_{3}}A_{m_{1},m_{3},m_{2},q}(a,c,b)\nonumber\\
&&\quad+B_{m_{1},m_{2},m_{3},q}(a,b,c)+(-1)^{m_{2}+m_{3}}B_{m_{1},m_{3},m_{2},q}(a,c,b)\nonumber\\
&&\quad+C_{m_{1},m_{2},m_{3},q}(a,b,c)+(-1)^{m_{1}}C_{m_{1},m_{3},m_{2},q}(a,c,b)\nonumber\\
&&\quad-D_{m_{1},m_{2},m_{3},q}(b,c)+D_{m_{1},m_{3},m_{2},q}(c,b)+C_{2},
\end{eqnarray}
where $R(m_{1},m_{2},q;a,b)$ is as in \eqref{eq2.2},
\begin{equation*}
A_{m_{1},m_{2},m_{3},q}(a,b,c),B_{m_{1},m_{2},m_{3},q}(a,b,c),C_{m_{1},m_{2},m_{3},q}(a,b,c),D_{m_{1},m_{2},m_{3},q}(b,c),C_{2}
\end{equation*}
are as in \eqref{eq2.15}.
\end{lemma}

\begin{proof}
Replacing $m$ by $m_{2}$, $n$ by $m_{3}$, $a$ by $b$, $b$ by $c$, $x$ by $t/d$, $y$ by $0$, $z$ by $cs/d$, and then making the operation $\sum_{t=0}^{d-1}$ on both sides of \eqref{eq3.1}, we discover from \eqref{eq2.4} that for $m_{2},m_{3},b,c,d\in\mathbb{N}$, $s\in\mathbb{Z}$ with $m_{3}\geq2$,
\begin{eqnarray}\label{eq3.5}
&&\sum_{t=0}^{d-1}\overline{B}_{m_{2}}\biggl(\frac{bt}{d}\biggl)\overline{B}_{m_{3}}\biggl(\frac{ct}{d}+\frac{cs}{d}\biggl)\nonumber\\
&&=m_{3}c^{m_{3}-1}\sum_{j=0}^{m_{2}}\binom{m_{2}}{j}\frac{(-1)^{j}b^{m_{2}-j}}{d^{m_{2}+m_{3}-j-1}(m_{2}+m_{3}-j)}
\sum_{l=1}^{c}\overline{B}_{j}\biggl(\frac{bl}{c}+\frac{bs}{d}\biggl)\nonumber\\
&&\qquad\times\overline{B}_{m_{2}+m_{3}-j}\biggl(\frac{dl}{c}\biggl)\nonumber\\
&&\quad+m_{2}b^{m_{2}-1}\sum_{j=0}^{m_{3}}\binom{m_{3}}{j}\frac{(-1)^{j}c^{m_{3}-j}}{d^{m_{2}+m_{3}-j-1}(m_{2}+m_{3}-j)}
\sum_{l=1}^{b}\overline{B}_{j}\biggl(\frac{cl}{b}-\frac{cs}{d}\biggl)\nonumber\\
&&\qquad\times\overline{B}_{m_{2}+m_{3}-j}\biggl(\frac{dl}{b}\biggl)\nonumber\\
&&\quad+\frac{(-1)^{m_{3}-1}dm_{2}!m_{3}!(b,c)^{m_{2}+m_{3}}\overline{B}_{m_{2}+m_{3}}\bigl(-\frac{\frac{bc}{(b,c)}s}{d}\bigl)}{b^{m_{3}}c^{m_{2}}(m_{2}+m_{3})!}.
\end{eqnarray}
It follows from \eqref{eq2.5}, \eqref{eq2.11} and \eqref{eq3.5} that for $m_{1},m_{2},m_{3},a,b,c,d\in\mathbb{N}$, $s\in\mathbb{Z}$ with $m_{3}\geq2$ and $(a,d)=(b,d)=(c,d)=1$,
\begin{eqnarray*}
&&\sum_{s=0}^{d-1}\sum_{t=0}^{d-1}\overline{B}_{m_{1}}\biggl(\frac{as}{d}\biggl)\overline{B}_{m_{2}}
\biggl(\frac{bt}{d}\biggl)\overline{B}_{m_{3}}\biggl(\frac{c(s+t)}{d}\biggl)\\
&&=m_{3}c^{m_{3}-1}\sum_{j=1}^{m_{2}}\binom{m_{2}}{j}\frac{(-1)^{j}b^{m_{2}-j}}{d^{m_{2}+m_{3}-j-1}(m_{2}+m_{3}-j)}
\sum_{l=1}^{c}\overline{B}_{m_{2}+m_{3}-j}\biggl(\frac{dl}{c}\biggl)\\
&&\qquad\times\sum_{s=0}^{d-1}\overline{B}_{m_{1}}\biggl(\frac{as}{d}\biggl)\overline{B}_{j}\biggl(\frac{bs}{d}+\frac{bl}{c}\biggl)\nonumber\\
&&\quad+m_{2}b^{m_{2}-1}\sum_{j=1}^{m_{3}}\binom{m_{3}}{j}\frac{c^{m_{3}-j}}{d^{m_{2}+m_{3}-j-1}(m_{2}+m_{3}-j)}
\sum_{l=1}^{b}\overline{B}_{m_{2}+m_{3}-j}\biggl(\frac{dl}{b}\biggl)\\
&&\qquad\times\sum_{s=0}^{d-1}\overline{B}_{m_{1}}\biggl(\frac{as}{d}\biggl)\overline{B}_{j}\biggl(\frac{cs}{d}-\frac{cl}{b}\biggl)\nonumber\\
&&\quad+\frac{(-1)^{m_{2}-1}dm_{2}!m_{3}!(b,c)^{m_{2}+m_{3}}}{b^{m_{3}}c^{m_{2}}(m_{2}+m_{3})!}\sum_{s=0}^{d-1}\overline{B}_{m_{1}}
\biggl(\frac{as}{d}\biggl)\overline{B}_{m_{2}+m_{3}}\biggl(\frac{\frac{bc}{(b,c)}s}{d}\biggl)\\
&&\quad+\frac{m_{3}b^{m_{2}}\overline{B}_{m_{1}}(0)B_{m_{2}+m_{3}}}{c^{m_{2}}d^{m_{1}+m_{2}+m_{3}-2}(m_{2}+m_{3})}+\frac{m_{2}c^{m_{3}}
\overline{B}_{m_{1}}(0)B_{m_{2}+m_{3}}}{b^{m_{3}}d^{m_{1}+m_{2}+m_{3}-2}(m_{2}+m_{3})},
\end{eqnarray*}
which together with \eqref{eq2.6} means that for $q,m_{1},m_{2},m_{3},a,b,c\in\mathbb{N}$ with $q,m_{3}\geq2$ and $(a,q)=(b,q)=(c,q)=1$,
\begin{eqnarray}\label{eq3.6}
&&\sum_{d\mid q}\mu\biggl(\frac{q}{d}\biggl)d^{m_{1}+m_{2}+m_{3}-2}\sum_{s=0}^{d-1}\sum_{t=0}^{d-1}\overline{B}_{m_{1}}\biggl(\frac{as}{d}\biggl)\overline{B}_{m_{2}}
\biggl(\frac{bt}{d}\biggl)\overline{B}_{m_{3}}\biggl(\frac{c(s+t)}{d}\biggl)\nonumber\\
&&=m_{3}c^{m_{3}-1}\sum_{j=1}^{m_{2}}\binom{m_{2}}{j}\frac{(-1)^{j}b^{m_{2}-j}}{m_{2}+m_{3}-j}\sum_{d\mid q}\mu\biggl(\frac{q}{d}\biggl)d^{m_{1}+j-1}
\sum_{l=1}^{c}\overline{B}_{m_{2}+m_{3}-j}\biggl(\frac{dl}{c}\biggl)\nonumber\\
&&\qquad\times\sum_{s=0}^{d-1}\overline{B}_{m_{1}}\biggl(\frac{as}{d}\biggl)\overline{B}_{j}\biggl(\frac{bs}{d}+\frac{bl}{c}\biggl)\nonumber\\
&&\quad+m_{2}b^{m_{2}-1}\sum_{j=1}^{m_{3}}\binom{m_{3}}{j}\frac{c^{m_{3}-j}}{m_{2}+m_{3}-j}\sum_{d\mid q}\mu\biggl(\frac{q}{d}\biggl)d^{m_{1}+j-1}
\sum_{l=1}^{b}\overline{B}_{m_{2}+m_{3}-j}\biggl(\frac{dl}{b}\biggl)\nonumber\\
&&\qquad\times\sum_{s=0}^{d-1}\overline{B}_{m_{1}}\biggl(\frac{as}{d}\biggl)\overline{B}_{j}\biggl(\frac{cs}{d}-\frac{cl}{b}\biggl)\nonumber\\
&&\quad+\frac{(-1)^{m_{2}-1}m_{2}!m_{3}!(b,c)^{m_{2}+m_{3}}}{b^{m_{3}}c^{m_{2}}(m_{2}+m_{3})!}\sum_{d\mid q}\mu\biggl(\frac{q}{d}\biggl)d^{m_{1}+m_{2}+m_{3}-1}\nonumber\\
&&\qquad\times\sum_{s=0}^{d-1}\overline{B}_{m_{1}}
\biggl(\frac{as}{d}\biggl)\overline{B}_{m_{2}+m_{3}}\biggl(\frac{\frac{bc}{(b,c)}s}{d}\biggl).
\end{eqnarray}
It is easy to see from Lemma \ref{lem3.1}, \eqref{eq2.4} and \eqref{eq2.11} that for $m_{1},j,a,b,c,d\in\mathbb{N}$, $l\in \mathbb{Z}$ with $(a,d)=1$,
\begin{eqnarray}\label{eq3.7}
&&\sum_{s=0}^{d-1}\overline{B}_{m_{1}}\biggl(\frac{as}{d}\biggl)\overline{B}_{j}\biggl(\frac{bs}{d}+\frac{bl}{c}\biggl)\nonumber\\
&&=jb^{j-1}\sum_{j_{1}=0}^{m_{1}}\binom{m_{1}}{j_{1}}\frac{(-1)^{j_{1}}a^{m_{1}-j_{1}}}{d^{m_{1}+j-j_{1}-1}(m_{1}+j-j_{1})}
\sum_{l_{1}=1}^{b}\overline{B}_{j_{1}}\biggl(\frac{al_{1}}{b}+\frac{al}{c}\biggl)\nonumber\\
&&\qquad\times\overline{B}_{m_{1}+j-j_{1}}\biggl(\frac{dl_{1}}{b}+\frac{dl}{c}\biggl)\nonumber\\
&&\quad+m_{1}a^{m_{1}-1}\sum_{j_{1}=0}^{j}\binom{j}{j_{1}}\frac{(-1)^{j_{1}}b^{j-j_{1}}}{d^{m_{1}+j-j_{1}-1}(m_{1}+j-j_{1})}
\sum_{l_{1}=1}^{a}\overline{B}_{j_{1}}\biggl(\frac{bl_{1}}{a}-\frac{bl}{c}\biggl)\nonumber\\
&&\qquad\times\overline{B}_{m_{1}+j-j_{1}}\biggl(\frac{dl_{1}}{a}\biggl)\nonumber\\
&&\quad+\frac{(-1)^{m_{1}-1}dm_{1}!j!(a,b)^{m_{1}+j}\overline{B}_{m_{1}+j}\bigl(\frac{\frac{ab}{(a,b)}l}{c}\bigl)}{a^{j}b^{m_{1}}(m_{1}+j)!}-\frac{\delta_{1,m_{1}}\delta_{1,j}}{4}\delta_{\mathbb{Z}}\biggl(\frac{bl}{c}\biggl).
\end{eqnarray}
By substituting $c$ for $b$, $b$ for $c$, $-l$ for $l$ in \eqref{eq3.7}, we get from \eqref{eq2.11} that
\begin{eqnarray}\label{eq3.8}
&&\sum_{s=0}^{d-1}\overline{B}_{m_{1}}\biggl(\frac{as}{d}\biggl)\overline{B}_{j}\biggl(\frac{cs}{d}-\frac{cl}{b}\biggl)\nonumber\\
&&=jc^{j-1}\sum_{j_{1}=0}^{m_{1}}\binom{m_{1}}{j_{1}}\frac{(-1)^{j_{1}}a^{m_{1}-j_{1}}}{d^{m_{1}+j-j_{1}-1}(m_{1}+j-j_{1})}
\sum_{l_{1}=1}^{c}\overline{B}_{j_{1}}\biggl(\frac{al_{1}}{c}-\frac{al}{b}\biggl)\nonumber\\
&&\qquad\times\overline{B}_{m_{1}+j-j_{1}}\biggl(\frac{dl_{1}}{c}-\frac{dl}{b}\biggl)\nonumber\\
&&\quad+m_{1}a^{m_{1}-1}\sum_{j_{1}=0}^{j}\binom{j}{j_{1}}\frac{(-1)^{j_{1}}c^{j-j_{1}}}{d^{m_{1}+j-j_{1}-1}(m_{1}+j-j_{1})}
\sum_{l_{1}=1}^{a}\overline{B}_{j_{1}}\biggl(\frac{cl_{1}}{a}+\frac{cl}{b}\biggl)\nonumber\\
&&\qquad\times\overline{B}_{m_{1}+j-j_{1}}\biggl(\frac{dl_{1}}{a}\biggl)\nonumber\\
&&\quad+\frac{(-1)^{j-1}dm_{1}!j!(a,c)^{m_{1}+j}\overline{B}_{m_{1}+j}\bigl(\frac{\frac{ac}{(a,c)}l}{b}\bigl)}{a^{j}c^{m_{1}}(m_{1}+j)!}-\frac{\delta_{1,m_{1}}\delta_{1,j}}{4}\delta_{\mathbb{Z}}\biggl(\frac{cl}{b}\biggl).
\end{eqnarray}
Taking $l=0$ and then replacing $j$ by $m_{2}+m_{3}$, $b$ by $bc/(b,c)$ in \eqref{eq3.7}, we have
\begin{eqnarray}\label{eq3.9}
&&\sum_{s=0}^{d-1}\overline{B}_{m_{1}}\biggl(\frac{as}{d}\biggl)\overline{B}_{m_{2}+m_{3}}\biggl(\frac{\frac{bc}{(b,c)}s}{d}\biggl)\nonumber\\
&&=(m_{2}+m_{3})\biggl(\frac{bc}{(b,c)}\biggl)^{m_{2}+m_{3}-1}\sum_{j_{1}=0}^{m_{1}}\binom{m_{1}}{j_{1}}
\frac{(-1)^{j_{1}}a^{m_{1}-j_{1}}}{d^{m_{1}+m_{2}+m_{3}-j_{1}-1}}\nonumber\\
&&\qquad\times\frac{1}{m_{1}+m_{2}+m_{3}-j_{1}}\sum_{l_{1}=1}^{\frac{bc}{(b,c)}}\overline{B}_{j_{1}}\biggl(\frac{al_{1}}{\frac{bc}{(b,c)}}\biggl)
\overline{B}_{m_{1}+m_{2}+m_{3}-j_{1}}\biggl(\frac{dl_{1}}{\frac{bc}{(b,c)}}\biggl)\nonumber\\
&&\quad+m_{1}a^{m_{1}-1}\sum_{j_{1}=0}^{m_{2}+m_{3}}\binom{m_{2}+m_{3}}{j_{1}}
\frac{(-1)^{j_{1}}\bigl(\frac{bc}{(b,c)}\bigl)^{m_{2}+m_{3}-j_{1}}}{d^{m_{1}+m_{2}+m_{3}-j_{1}-1}(m_{1}+m_{2}+m_{3}-j_{1})}\nonumber\\
&&\qquad\times\sum_{l_{1}=1}^{a}\overline{B}_{j_{1}}\biggl(\frac{\frac{bc}{(b,c)}l_{1}}{a}\biggl)\overline{B}_{m_{1}+m_{2}+m_{3}-j_{1}}\biggl(\frac{dl_{1}}{a}\biggl)\nonumber\\
&&\quad+\frac{(-1)^{m_{1}-1}dm_{1}!(m_{2}+m_{3})!\bigl(a,\frac{bc}{(b,c)}\bigl)^{m_{1}+m_{2}+m_{3}}B_{m_{1}+m_{2}+m_{3}}}{a^{m_{2}+m_{3}}\bigl(\frac{bc}{(b,c)}\bigl)^{m_{1}}(m_{1}+m_{2}+m_{3})!}.
\end{eqnarray}
Thus, inserting \eqref{eq3.7}, \eqref{eq3.8} and \eqref{eq3.9} into \eqref{eq3.6}, in view of \eqref{eq2.11} and the property of residue
systems used in \eqref{eq2.5}, we prove Lemma \ref{lem3.3}.
\end{proof}

\section{Proofs of Theorems \ref{thm2.1} and \ref{thm2.4}}

\noindent{\em{The proof of Theorem \ref{thm2.1}.}} It is easily seen that for $m\in\mathbb{N}$,
\begin{eqnarray}\label{eq4.1}
\sideset{}{'}\sum_{l=-\infty}^{+\infty}\frac{\chi(l)}{l^{m}}&=&\sum_{l=1}^{\infty}\frac{\chi(l)}{l^{m}}+\sum_{l=1}^{\infty}\frac{\chi(-l)}{(-l)^{m}}\nonumber\\
&=&\bigl(1+(-1)^{m}\chi(-1)\bigl)L(m,\chi),
\end{eqnarray}
where the dash denotes throughout that undefined terms are excluded from the sum. We know from \eqref{eq4.1} that for $m,n\in\mathbb{N}$,
\begin{eqnarray}\label{eq4.2}
&&\biggl(\sideset{}{'}\sum_{l=-\infty}^{+\infty}\frac{\chi(l)}{l^{m}}\biggl)\biggl(\sideset{}{'}\sum_{j=-\infty}^{+\infty}\frac{\overline{\chi}(j)}{j^{n}}\biggl)\nonumber\\
&&=\bigl(1+(-1)^{m}\chi(-1)\bigl)\bigl(1+(-1)^{n}\overline{\chi}(-1)\bigl)L(m,\chi)L(n,\overline{\chi}).
\end{eqnarray}
It follows from \eqref{eq4.2} that for $q,m,n\in\mathbb{N}$, $a,b\in\mathbb{Z}$ with $(a,q)=(b,q)=1$,
\begin{eqnarray}\label{eq4.3}
V_{\chi,q}(m,n;a,b)&=&\frac{1}{4}\sum_{\text{$\chi$ (mod $q$)}}\chi(a)\overline{\chi}(b)\biggl(\sideset{}{'}\sum_{l=-\infty}^{+\infty}\frac{\chi(l)}{l^{m}}\biggl)
\biggl(\sideset{}{'}\sum_{j=-\infty}^{+\infty}\frac{\overline{\chi}(j)}{j^{n}}\biggl)\nonumber\\
&=&\frac{1}{4}\sideset{}{'}\sum_{l=-\infty}^{+\infty}\sideset{}{'}\sum_{j=-\infty}^{+\infty}\frac{1}{l^{m}j^{n}}\sum_{\text{$\chi$ (mod $q$)}}\chi(al)\overline{\chi}(bj)\nonumber\\
&=&\frac{1}{4}\underset{(j,q)=1}{\sideset{}{'}\sum_{l=-\infty}^{+\infty}\sideset{}{'}\sum_{j=-\infty}^{+\infty}}\frac{1}{l^{m}j^{n}}\sum_{\text{$\chi$ (mod $q$)}}\chi(al)\overline{\chi}(bj).
\end{eqnarray}
We now apply the orthogonality relation for the Dirichlet characters modulo $q$ stated in \cite[Theorem 6.16]{apostol} to rewrite \eqref{eq4.3} as
\begin{equation}\label{eq4.4}
V_{\chi,q}(m,n;a,b)
=\frac{\phi(q)}{4}\underset{\substack{(j,q)=1\\\text{$al\equiv bj$ (mod $q$)}}}{\sideset{}{'}\sum_{l=-\infty}^{+\infty}\sideset{}{'}\sum_{j=-\infty}^{+\infty}}\frac{1}{l^{m}j^{n}}.
\end{equation}
Note that from \eqref{eq2.6} and the familiar geometric sum shown in \cite[Theorem 8.1]{apostol} we have
\begin{eqnarray}\label{eq4.5}
\underset{\substack{(j,q)=1\\\text{$al\equiv bj$ (mod $q$)}}}{\sideset{}{'}\sum_{l=-\infty}^{+\infty}\sideset{}{'}\sum_{j=-\infty}^{+\infty}}\frac{1} {l^{m}j^{n}}
&=&\frac{1}{q}\underset{(j,q)=1}{\sideset{}{'}\sum_{l=-\infty}^{+\infty}\sideset{}{'}\sum_{j=-\infty}^{+\infty}}\frac{1}{l^{m}j^{n}}
\sum_{r=0}^{q-1}e^{\frac{2\pi\mathrm{i}(al-bj)r}{q}}\nonumber\\
&=&\frac{1}{q}\sum_{d\mid q}\mu(d)\underset{d\mid j}{\sideset{}{'}\sum_{l=-\infty}^{+\infty}\sideset{}{'}\sum_{j=-\infty}^{+\infty}}\frac{1}{l^{m}j^{n}}\sum_{r=0}^{q-1}e^{\frac{2\pi\mathrm{i}(al-bj)r}{q}}\nonumber\\
&=&\frac{1}{q}\sum_{d\mid q}\frac{\mu(d)}{d^{n}}\sideset{}{'}\sum_{l=-\infty}^{+\infty}\sideset{}{'}\sum_{j=-\infty}^{+\infty}\frac{1}{l^{m}j^{n}}
\sum_{r=0}^{q-1}e^{\frac{2\pi\mathrm{i}(al-bdj)r}{q}}\nonumber\\
&=&\frac{1}{q}\sum_{d\mid q}\frac{\mu(d)}{d^{n}}\sum_{r=0}^{q-1}{\sideset{}{'}\sum_{l=-\infty}^{+\infty}\frac{e^{\frac{2\pi\mathrm{i}arl}{q}}}{l^{m}}
\sideset{}{'}\sum_{j=-\infty}^{+\infty}}\frac{e^{-\frac{2\pi\mathrm{i}bdrj}{q}}}{j^{n}}\nonumber\\
&=&\frac{(-1)^{n}}{q}\sum_{d\mid q}\frac{\mu(d)}{d^{n}}\sum_{r=0}^{q-1}{\sideset{}{'}\sum_{l=-\infty}^{+\infty}\frac{e^{\frac{2\pi\mathrm{i}arl}{q}}}{l^{m}}
\sideset{}{'}\sum_{j=-\infty}^{+\infty}}\frac{e^{\frac{2\pi\mathrm{i}bdrj}{q}}}{j^{n}}.
\end{eqnarray}
Since for $n\in\mathbb{N}$, $x\in\mathbb{R}$, the Bernoulli function $\overline{B}_{n}(x)$ defined in \eqref{eq2.1} can be given by the Fourier series (see, e.g., \cite[Theorem 12.19]{apostol} or \cite[Equation (2)]{hall}),
\begin{equation}\label{eq4.6}
\overline{B}_{n}(x)=-\frac{n!}{(2\pi\mathrm{i})^{n}}\sideset{}{'}\sum_{k=-\infty}^{+\infty}\frac{e^{2\pi\mathrm{i}kx}}{k^{n}},
\end{equation}
by \eqref{eq4.5} and \eqref{eq4.6} we have
\begin{equation}\label{eq4.7}
\underset{\substack{(j,q)=1\\\text{$al\equiv bj$ (mod $q$)}}}{\sideset{}{'}\sum_{l=-\infty}^{+\infty}\sideset{}{'}\sum_{j=-\infty}^{+\infty}}\frac{1} {l^{m}j^{n}}
=\frac{(-1)^{n}(2\pi \mathrm{i})^{m+n}}{qm!n!}\sum_{d\mid q}\frac{\mu(d)}{d^{n}}\sum_{r=0}^{q-1}\overline{B}_{m}\biggl(\frac{ar}{q}\biggl)\overline{B}_{n}\biggl(\frac{bdr}{q}\biggl).
\end{equation}
Inserting \eqref{eq4.7} into \eqref{eq4.4}, we obtain that for $q,m,n\in\mathbb{N}$, $a,b\in\mathbb{Z}$ with $(a,q)=(b,q)=1$,
\begin{equation}\label{eq4.8}
V_{\chi,q}(m,n;a,b)=\frac{(-1)^{\frac{m-n}{2}}(2\pi)^{m+n}\phi(q)}{4q m!n!}\sum_{d\mid q}\frac{\mu(d)}{d^{n}}\sum_{r=0}^{q-1}\overline{B}_{m}\biggl(\frac{ar}{q}\biggl)\overline{B}_{n}\biggl(\frac{bdr}{q}\biggl).
\end{equation}
With the help of the division algorithm described in \cite[Theorem 1.14]{apostol}, we discover from \eqref{eq2.5} that the sum on the right hand side of \eqref{eq4.8} can be rewritten as
\begin{eqnarray}\label{eq4.9}
\sum_{r=0}^{q-1}\overline{B}_{m}\biggl(\frac{ar}{q}\biggl)\overline{B}_{n}\biggl(\frac{bdr}{q}\biggl)
&=&\sum_{l=0}^{d-1}\sum_{j=0}^{\frac{q}{d}-1}\overline{B}_{m}\biggl(\frac{a(\frac{q}{d}l+j)}{q}\biggl)\overline{B}_{n}\biggl(\frac{bd(\frac{q}{d}l+j)}{q}\biggl)\nonumber\\
&=&\sum_{l=0}^{d-1}\sum_{j=0}^{\frac{q}{d}-1}\overline{B}_{m}\biggl(\frac{a(\frac{q}{d}l+j)}{q}\biggl)\overline{B}_{n}\biggl(\frac{bj}{\frac{q}{d}}\biggl)\nonumber\\
&=&\sum_{l=0}^{d-1}\sum_{j=0}^{\frac{q}{d}-1}\overline{B}_{m}\biggl(\frac{al}{d}+\frac{aj}{q}\biggl)\overline{B}_{n}\biggl(\frac{bj}{\frac{q}{d}}\biggl)\nonumber\\
&=&\frac{1}{d^{m-1}}\sum_{j=0}^{\frac{q}{d}-1}\overline{B}_{m}\biggl(\frac{aj}{\frac{q}{d}}\biggl)\overline{B}_{n}\biggl(\frac{bj}{\frac{q}{d}}\biggl).
\end{eqnarray}
By combining \eqref{eq4.8} and \eqref{eq4.9}, we conclude that for $q,m,n\in\mathbb{N}$, $a,b\in\mathbb{Z}$ with $(a,q)=(b,q)=1$,
\begin{eqnarray}\label{eq4.10}
V_{\chi,q}(m,n;a,b)&=&\frac{(-1)^{\frac{m-n}{2}}(2\pi)^{m+n}\phi(q)}{4q m!n!}\sum_{d\mid q}\frac{\mu(d)}{d^{m+n-1}}\sum_{j=0}^{\frac{q}{d}-1}\overline{B}_{m}\biggl(\frac{aj}{\frac{q}{d}}\biggl)\overline{B}_{n}\biggl(\frac{bj}{\frac{q}{d}}\biggl)\nonumber\\
&=&\frac{(-1)^{\frac{m-n}{2}}(2\pi)^{m+n}\phi(q)}{4q^{m+n} m!n!}\sum_{d\mid q}\mu\biggl(\frac{q}{d}\biggl)d^{m+n-1}\nonumber\\
&&\times\sum_{j=0}^{d-1}\overline{B}_{m}\biggl(\frac{aj}{d}\biggl)\overline{B}_{n}\biggl(\frac{bj}{d}\biggl).
\end{eqnarray}
Now we get \eqref{eq2.2} immediately when applying Lemma \ref{lem3.2} to the right hand side of \eqref{eq4.10}. This completes the proof of Theorem \ref{thm2.1}.

\noindent{\em{The proof of Theorem \ref{thm2.4}.}} We know from \eqref{eq4.1} that for $m_{1},m_{2},m_{3}\in\mathbb{N}$,
\begin{eqnarray}\label{eq4.11}
&&\biggl(\sideset{}{'}\sum_{l=-\infty}^{+\infty}\frac{\chi_{1}(l)}{l^{m_{1}}}\biggl)\biggl(\sideset{}{'}\sum_{j=-\infty}^{+\infty}\frac{\chi_{2}(j)}{j^{m_{2}}}\biggl)
\biggl(\sideset{}{'}\sum_{k=-\infty}^{+\infty}\frac{\overline{\chi_{1}\chi_{2}}(k)}{k^{m_{3}}}\biggl)\nonumber\\
&&=\bigl(1+(-1)^{m_{1}}\chi_{1}(-1)\bigl)\bigl(1+(-1)^{m_{2}}\chi_{2}(-1)\bigl)\bigl(1+(-1)^{m_{3}}\overline{\chi_{1}\chi_{2}}(-1)\bigl)\nonumber\\
&&\quad\times L(m_{1},\chi_{1})L(m_{2},\chi_{2})L(m_{3},\overline{\chi_{1}\chi_{2}}).
\end{eqnarray}
It follows from \eqref{eq4.11} that for $m_{1},m_{2},m_{3}\in\mathbb{N}$, $a,b,c\in\mathbb{Z}$ with $m_{1}+m_{2}\equiv m_{3}$ (mod $2$) and $(a,q)=(b,q)=(c,q)=1$,
\begin{eqnarray}\label{eq4.12}
&&V_{\chi_{1},\chi_{2},q}(m_{1},m_{2},m_{3};a,b,c)\nonumber\\
&&=\frac{1}{8}\sum_{\text{$\chi_{1},\chi_{2}$ (mod $q$)}}\biggl(\sideset{}{'}\sum_{l=-\infty}^{+\infty}\frac{\chi_{1}(al)}{l^{m_{1}}}\biggl)
\biggl(\sideset{}{'}\sum_{j=-\infty}^{+\infty}\frac{\chi_{2}(bj)}{j^{m_{2}}}\biggl)
\biggl(\sideset{}{'}\sum_{k=-\infty}^{+\infty}\frac{\overline{\chi_{1}\chi_{2}}(ck)}{k^{m_{3}}}\biggl)\nonumber\\
&&=\frac{1}{8}\underset{(k,q)=1}{\sideset{}{'}\sum_{l=-\infty}^{+\infty}\sideset{}{'}\sum_{j=-\infty}^{+\infty}
\sideset{}{'}\sum_{k=-\infty}^{+\infty}}\frac{1}{l^{m_{1}}j^{m_{2}}k^{m_{3}}}\nonumber\\
&&\quad\times\sum_{\text{$\chi_{1}$ (mod $q$)}}\chi_{1}(al)\overline{\chi_{1}}(ck)\sum_{\text{$\chi_{2}$ (mod $q$)}}\chi_{2}(bj)\overline{\chi_{2}}(ck).
\end{eqnarray}
Using the orthogonality relation for the Dirichlet characters modulo $q$, we rewrite \eqref{eq4.12} as
\begin{equation}\label{eq4.13}
V_{\chi_{1},\chi_{2},q}(m_{1},m_{2},m_{3};a,b,c)=\frac{\phi(q)^{2}}{8}\underset{\substack{(k,q)=1\\ \text{$al\equiv ck$ (mod $q$)}\\ \text{$bj\equiv ck$ (mod $q$)}}}{\sideset{}{'}\sum_{l=-\infty}^{+\infty}\sideset{}{'}\sum_{j=-\infty}^{+\infty}\sideset{}{'}\sum_{k=-\infty}^{+\infty}}\frac{1}{l^{m_{1}}j^{m_{2}}k^{m_{3}}}.
\end{equation}
In a similar consideration to \eqref{eq4.7}, we have
\begin{eqnarray*}
&&\underset{\substack{(k,q)=1\\ \text{$al\equiv ck$ (mod $q$)}\\ \text{$bj\equiv ck$ (mod $q$)}}}{\sideset{}{'}\sum_{l=-\infty}^{+\infty}\sideset{}{'}\sum_{j=-\infty}^{+\infty}\sideset{}{'}\sum_{k=-\infty}^{+\infty}}\frac{1}{l^{m_{1}}j^{m_{2}}k^{m_{3}}}\\
&&=\frac{1}{q^{2}}\underset{(k,q)=1}{\sideset{}{'}\sum_{l=-\infty}^{+\infty}\sideset{}{'}\sum_{j=-\infty}^{+\infty}
\sideset{}{'}\sum_{k=-\infty}^{+\infty}}\frac{1}{l^{m_{1}}j^{m_{2}}k^{m_{3}}}\sum_{r=0}^{q-1}e^{\frac{2\pi\mathrm{i}(al-ck)r}{q}}
\sum_{s=0}^{q-1}e^{\frac{2\pi\mathrm{i}(bj-ck)s}{q}}\\
&&=\frac{1}{q^{2}}\sum_{d\mid q}\mu(d)\underset{d\mid k}{\sideset{}{'}\sum_{l=-\infty}^{+\infty}\sideset{}{'}\sum_{j=-\infty}^{+\infty}\sideset{}{'}\sum_{k=-\infty}^{+\infty}}\frac{1}{l^{m_{1}}j^{m_{2}}k^{m_{3}}}
\sum_{r=0}^{q-1}e^{\frac{2\pi\mathrm{i}(al-ck)r}{q}}\sum_{s=0}^{q-1}e^{\frac{2\pi\mathrm{i}(bj-ck)s}{q}}\\
&&=-\frac{(-1)^{m_{3}}(2\pi \mathrm{i})^{m_{1}+m_{2}+m_{3}}}{q^{2}m_{1}! m_{2}!m_{3}!}\sum_{d\mid q}\frac{\mu(d)}{d^{m_{3}}}\sum_{r=0}^{q-1}\sum_{s=0}^{q-1}\overline{B}_{m_{1}}\biggl(\frac{ar}{q}\biggl)\overline{B}_{m_{2}}\biggl(\frac{bs}{q}\biggl)\nonumber\\
&&\quad\times\overline{B}_{m_{3}}\biggl(\frac{cd(r+s)}{q}\biggl),
\end{eqnarray*}
from which and \eqref{eq4.13} we conclude that for $m_{1},m_{2},m_{3}\in\mathbb{N}$, $a,b,c\in\mathbb{Z}$ with $m_{1}+m_{2}\equiv m_{3}$ (mod $2$) and $(a,q)=(b,q)=(c,q)=1$,
\begin{eqnarray}\label{eq4.14}
&&V_{\chi_{1},\chi_{2},q}(m_{1},m_{2},m_{3};a,b,c)\nonumber\\
&&=-\frac{(-1)^{\frac{m_{1}+m_{2}-m_{3}}{2}}(2\pi)^{m_{1}+m_{2}+m_{3}}\phi(q)^{2}}{8q^{2}m_{1}!m_{2}!m_{3}!}\sum_{d\mid q}\frac{\mu(d)}{d^{m_{3}}}\nonumber\\
&&\quad\times\sum_{r=0}^{q-1}\sum_{s=0}^{q-1}\overline{B}_{m_{1}}\biggl(\frac{ar}{q}\biggl)\overline{B}_{m_{2}}\biggl(\frac{bs}{q}\biggl)
\overline{B}_{m_{3}}\biggl(\frac{cd(r+s)}{q}\biggl).
\end{eqnarray}
Note that from \eqref{eq2.5} and the division algorithm, we can rewrite the sum on the right hand side of \eqref{eq4.14} as
\begin{eqnarray}\label{eq4.15}
&&\sum_{r=0}^{q-1}\sum_{s=0}^{q-1}\overline{B}_{m_{1}}\biggl(\frac{ar}{q}\biggl)\overline{B}_{m_{2}}\biggl(\frac{bs}{q}\biggl)
\overline{B}_{m_{3}}\biggl(\frac{cd(r+s)}{q}\biggl)\nonumber\\
&&=\sum_{l=0}^{d-1}\sum_{j=0}^{\frac{q}{d}-1}\sum_{k=0}^{d-1}\sum_{t=0}^{\frac{q}{d}-1}\overline{B}_{m_{1}}
\biggl(\frac{a(\frac{q}{d}l+j)}{q}\biggl)\overline{B}_{m_{2}}\biggl(\frac{b(\frac{q}{d}k+t)}{q}\biggl)\nonumber\\
&&\quad\times\overline{B}_{m_{3}}\biggl(\frac{cd\bigl(\frac{q}{d}(l+k)+j+t\bigl)}{q}\biggl)\nonumber\\
&&=\sum_{l=0}^{d-1}\sum_{j=0}^{\frac{q}{d}-1}\sum_{k=0}^{d-1}\sum_{t=0}^{\frac{q}{d}-1}
\overline{B}_{m_{1}}\biggl(\frac{al}{d}+\frac{aj}{q}\biggl)\overline{B}_{m_{2}}\biggl(\frac{bk}{d}+\frac{bt}{q}\biggl)
\overline{B}_{m_{3}}\biggl(\frac{c(j+t)}{\frac{q}{d}}\biggl)\nonumber\\
&&=\frac{1}{d^{m_{1}+m_{2}-2}}\sum_{j=0}^{\frac{q}{d}-1}\sum_{t=0}^{\frac{q}{d}-1}\overline{B}_{m_{1}}\biggl(\frac{aj}{\frac{q}{d}}\biggl)
\overline{B}_{m_{2}}\biggl(\frac{bt}{\frac{q}{d}}\biggl)\overline{B}_{m_{3}}\biggl(\frac{c(j+t)}{\frac{q}{d}}\biggl).
\end{eqnarray}
Inserting \eqref{eq4.15} into \eqref{eq4.14}, and it then follows that for $m_{1},m_{2},m_{3}\in\mathbb{N}$, $a,b,c\in\mathbb{Z}$ with $m_{1}+m_{2}\equiv m_{3}$ (mod $2$) and $(a,q)=(b,q)=(c,q)=1$,
\begin{eqnarray}\label{eq4.16}
&&V_{\chi_{1},\chi_{2},q}(m_{1},m_{2},m_{3};a,b,c)\nonumber\\
&&=-\frac{(-1)^{\frac{m_{1}+m_{2}-m_{3}}{2}}(2\pi)^{m_{1}+m_{2}+m_{3}}\phi(q)^{2}}{8q^{m_{1}+m_{2}+m_{3}}m_{1}!m_{2}!m_{3}!}\sum_{d\mid q}\mu\biggl(\frac{q}{d}\biggl)d^{m_{1}+m_{2}+m_{3}-2}\nonumber\\
&&\quad\times\sum_{s=0}^{d-1}\sum_{t=0}^{d-1}\overline{B}_{m_{1}}\biggl(\frac{as}{d}\biggl)\overline{B}_{m_{2}}\biggl(\frac{bt}{d}\biggl)
\overline{B}_{m_{3}}\biggl(\frac{c(s+t)}{d}\biggl).
\end{eqnarray}
Therefore, by applying Lemma \ref{lem3.3} to the right hand side of \eqref{eq4.16}, we get \eqref{eq2.15} and finish the proof of Theorem \ref{thm2.4}.

\end{document}